\documentclass[10pt]{article}%

\usepackage[utf8]{inputenc}

\usepackage{amsmath}
\usepackage{amsfonts}
\usepackage{mathrsfs}
\usepackage{amssymb, color}
\usepackage[linkcolor=black,anchorcolor=black,citecolor=black]{hyperref}
\usepackage[numbers,sort&compress]{natbib}
\usepackage{graphicx}
\numberwithin{equation}{section}
\usepackage[body={15.5cm,21cm}, top=3cm]{geometry}%
\providecommand{\U}[1]{\protect\rule{.1in}{.1in}}
\providecommand{\U}[1]{\protect \rule{.1in}{.1in}}
\newtheorem{theorem}{Theorem}[section]

\newtheorem{definition}[theorem]{Definition}

\newtheorem{lemma}[theorem]{Lemma}

\newtheorem{proposition}[theorem]{Proposition}
\newtheorem{remark}[theorem]{Remark}

\newtheorem{assumption}[theorem]{Assumption}
\newenvironment{proof}[1][Proof]{\noindent \textbf{#1.} }{\  \rule{0.5em}{0.5em}}

\def \E{\mathbf{E}}

\def \P{\mathbf{P}}

\usepackage [numbers,sort&compress] {natbib}

\usepackage{graphicx} 

\title{Mean-field BSDEs with non-Lipschitz coefficients and double mean reflections}

\begin{document}

	\author{ 	Hanwu Li \thanks{Research Center for Mathematics and Interdisciplinary Sciences, Shandong University, Qingdao 266237, Shandong, China. lihanwu@sdu.edu.cn.}
	\thanks{Frontiers Science Center for Nonlinear Expectations (Ministry of Education), Shandong University, Qingdao 266237, Shandong, China.}
    \and Jin Shi \thanks{Research Center for Mathematics and Interdisciplinary Sciences, Shandong University, Qingdao 266237, Shandong, China. }}
\maketitle
\begin{abstract}
     The present paper is devoted to the study of mean-field backward stochastic differential equations ($\text{MFBSDEs}$) with double mean reflections whose generators are not Lipschitz continuous. With the help of the Skorokhod problem and some a priori estimates for MFBSDEs, we establish the existence and uniqueness results for doubly mean reflected $\text{MFBSDEs}$.
	\end{abstract}
    

\textbf{MSC-classification}: 60H10
\section{Introduction}
 In recent years, motivated by the models of large stochastic particle systems with mean field interactions, Buckdahn, Djehiche, Li and Peng \cite{BDLP} studied the mean-field problem in a purely stochastic approach and deduced a new kind of BSDEs, called the mean-field BSDEs (MFBSDEs), which evolve according to the following equation:
 $$Y_t=\xi+\int_{t}^{T}f(s,Y_s,\P_{Y_s},Z_s,\P_{Z_s})ds- \int_t^T Z_s dB_s.$$ Roughly speaking, the generators of such type of BSDEs depend on the distributions of the solutions. They proved the existence and uniqueness of MFBSDEs when the generator $f$ is uniformly Lipschitz continuous in $(y,z)$ and the terminal condition $\xi$ is square-integrable.  Since then, the theory of MFBSDEs has made a rapid development due to its importance in applications, for example, in stochastic control (see, \cite{BDL,DTT,LJ}), partial differential equations (see, \cite{BLP,HHTW,LJ2}). Based on the  mean field theory and the reflected BSDE,  Li \cite{LJ1} studied  the following constrained MFBSDE
\begin{equation*}
 Y_t=\xi+\int_t^T \E\left[f(s,y_s,z_s,Y_s)\right]|_{y=Y_s,z=Z_s}ds-\int_t^T Z_s dB_s+(K_T-K_t),
\end{equation*}
 where the first component of the solution is forced to remain above a specified process $L$, referred to as the obstacle. That is, for any $t\in[0,T]$, we should make sure that  $Y_t\geq L_t$. The  objective is to determine the minimal solution $Y$, which is uniquely characterized by the Skorokhod condition 
 $$\int_{0}^{T} (Y_t-L_t) dK_t=0,$$ 
 where $K$ is a nondecreasing predictable process and provides a corrective force to maintain the solution $Y$ above the obstacle $L$. 
 Different from the pointwise constraints on the solutions, Hu, Moreau, and Wang \cite{HMW} studied the MFBSDEs with mean reflection, where the constraints are made on the distribution of the solution. More precisely, for any $t\in[0,T]$, the solution $Y$ is required to satisfy 
 $\mathbf{E}\left[l\left(t,Y_t\right)\right]\geq 0$. Here, $l$ is a nonlinear loss function, which is closely related to the risk measures (see \cite{BEH}). It is worth pointing out that for the mean reflected case, the force aiming to push the solution upward should be deterministic.  Recently, Li and Shi \cite{LS} investigated the doubly reflected case. 
 For more details on the mean reflected BSDEs, we refer the readers to \cite{CHM,FS,HHLLW,L,LU} and the references therein.   

It should be pointed out that the generators in all the above mentioned papers satisfy the Lipschitz assumption with respect to $y$, which may be somewhat restrictive in practical applications. For example, the utility function of the Kreps-Porteus’s type is not Lipschitz in $y$ (see \cite{DE}). 
Therefore, many works have been devoted to weakening the Lipschitz assumption about the $y$-term. To name a few, we may refer to \cite{BJ,FJ,M}. Especially, Mao \cite{M} proposed the following Mao's condition
$$\left|f(t,y,z)-f(t,y',z')\right|^2\leq \rho(\left|y-y'\right|^2)+\lambda^2\left|z-z'\right|^2,$$
where $\rho:[0,\infty)\rightarrow [0,\infty)$ is a continuous, non-decreasing and concave function satisfying certain integrable condition. 
Under the Mao condition, Cui and Zhao  studied the mean reflected BSDEs and MFBSDEs  in \cite{CZ} and \cite{CZ1}, respectively. They obtained the existence and uniqueness by using the Picard iteration method and the representation for the solutions to the Skorokhod problems.  

In this paper, we consider the existence and uniqueness of the solution to the following MFBSDE with double mean reflections:
\begin{equation}\label{BSDEDMR}
\begin{cases}
Y_t=\xi+\int_t^T f(s,Y_s,\P_{Y_s},Z_s,\P_{Z_s})ds-\int_t^T Z_s dB_s+K_T-K_t,\\
\mathbf{E}[L(t,Y_t)]\leq 0\leq \mathbf{E}[R(t,Y_t)], \ 0\leq t\leq T, \\
K_t=K^R_t-K^L_t,\ K^R,K^L\in I[0,T], \\
\int_0^T \mathbf{E}[R(t,Y_t)]dK_t^R=\int_0^T \mathbf{E}[L(t,Y_t)]dK^L_t=0,
\end{cases}
\end{equation}
where the generator $f$ satisfies Mao's condition. For this purpose, we construct a Picard iteration sequence and show that the triple of limit processes is the solution for \eqref{BSDEDMR}. In contrast to the mean reflection framework discussed in \cite{CZ}, the present study addresses the case of double mean reflections, wherein the governing equation is a mean-field backward stochastic differential equation (MFBSDE). Compared with \cite{CZ}, we may establish the global estimates directly by employing a priori estimates for MFBSDEs as developed in \cite{LLZ}. Concurrently, by integrating the oscillation estimate and the continuity estimate for the solution to the nonlinear Skorokhod problem presented in \cite{Li}, we derive some uniform estimates for the Picard iteration sequence. This approach enables us to demonstrate that the sequence is a Cauchy sequence under appropriate norm, thereby ensuring the well-posedness of Equation \eqref{BSDEDMR}.

This paper is organized as follows. In Section 2, we recall some basic results about the nonlinear Skorokhod problem. In Section 3, we first give several crucial a priori estimates and then prove the main result in this paper. Some a priori estimates for MFBSDEs are given in the Appendix.

\section{Preliminaries}
Given a fixed time horizon $T>0$, consider a complete probability space  $(\Omega, \mathcal{F}, \mathbf{P})$, on which $B=(B_t)_{0\leq t\leq T}$ is a standard $d$-dimensional Brownian motion. Let $(\mathcal{F}_t)_{0\leq t\leq T}$ be the natural filtration generated by $B$ augmented with all $\mathbb{P}$-null sets. Let us first introduce some frequently used notation in this paper. 
\begin{itemize}
    \item $\mathcal{L}^2(\mathcal{F}_t)$: the set of real-valued $\mathcal{F}_t$-measurable square integrable random variables for a fixed $t\in[0,T]$.
    \item $\mathcal{S}^2[u,v]$: the set of real-valued, continuous and progressively measurable processes $Y$ satisfying 
    $$\left \|Y  \right \|_{\mathcal{S}^2}:=\mathbf{E}\left[\sup_{t\in[u,v]}\left|Y_t\right|^{2}\right]^{\frac{1}{2}}<\infty.$$
    \item$\mathcal{H}^2[u,v]$: the set of $\mathbb{R}^d$-valued progressively measurable processes $Z$ satisfying 
    $$\left \|Z  \right \|_{\mathcal{H}^2}:=\mathbf{E}\left[\int_{u}^{v}\left|Z_t\right|^{2}dt\right]^{\frac{1}{2}}<\infty.$$
    
    \item $C\left[u,v\right]$: the set of continuous functions from $\left[u,v\right]$ to $\mathbb{R}$.

	\item  $BV\left[u,v\right]$: the set of functions $K\in C\left[u,v\right]$ with $K_u=0$ and $K$ is of bounded variation on $\left[u,v\right]$.
	
	\item  $I\left[u,v\right]$: the set of functions in $C\left[u,v\right]$ starting from the origin which is nondecreasing.
\end{itemize}
When the interval $[u,v]=[0,T]$, we always omit the time index. Throughout this paper, for a given set of parameters $\alpha$, $P(\alpha )$ will denote a constant only depending on these parameters and may change from line to line.

We first recall some basic results about the Skorokhod problem, which will be the building block to construct the solutions to doubly mean reflected MFBSDEs.
\begin{definition}\label{def1}
Let $s\in C[0,T]$,  and $l,r:[0,T] \times \mathbb{R}\rightarrow \mathbb{R}$ be two functions with $l\leq r$.  A pair of functions $(x,K)\in C[0,T]\times BV[0,T]$ is called a solution to the  Skorokhod problem for $s$ with nonlinear constraints $l,r$ ($(x,K)=\mathbb{SP}_l^r(s)$ for short) if 
\begin{itemize}
\item[(i)] $x_t=s_t+K_t$;
\item[(ii)] $l(t,x_t)\leq 0\leq r(t,x_t)$;
\item[(iii)]  $K_{0-}=0$ and $K$ has the decomposition $K=K^r-K^l$, where $K^r,K^l$ are nondecreasing functions satisfying  
\begin{align}
\int_0^{T} I_{\{l(s,x_s)<0\}}dK^l_s=0, \  \int_0^{T} I_{\{r(s,x_s)>0\}}dK^r_s=0.
\end{align}
\end{itemize}
\end{definition}

To obtain the result of existence and uniqueness for the Skorokhod problem, we propose the following assumption for the loss functions $l,r$.
\begin{assumption}\label{asslr}
\begin{itemize}
\item[(i)] For each fixed $x\in\mathbb{R}$, $l(\cdot,x),r(\cdot,x)\in C[0,T]$.
\item[(ii)] For any fixed $t\in [0,T]$, $l(t,\cdot)$, $r(t,\cdot)$ are strictly increasing.
\item[(iii)] There exist two positive constants $0<c<C<\infty$, such that for any $t\in [0,T]$ and $x,y\in \mathbb{R}$,
\begin{align*}
&c|x-y|\leq |l(t,x)-l(t,y)|\leq C|x-y|,\\
&c|x-y|\leq |r(t,x)-r(t,y)|\leq C|x-y|.
\end{align*} 
\item[(iv)] $\inf_{(t,x)\in[0,T]\times\mathbb{R}}(r(t,x)-l(t,x))>0$.
\end{itemize}
\end{assumption}
\begin{theorem}[\cite{Li}]\label{SP}
Suppose that $l,r$ satisfy Assumption \ref{asslr}. For any given $s\in C[0,T]$, there exists a unique pair of solution to the Skorokhod problem $(x,K)=\mathbb{SP}_l^r(s)$.
\end{theorem}


The following proposition provides the continuous dependence of the solution to the Skorokhod problem with respect to the input function $s$ and the loss functions $l,r$.
\begin{proposition}[\cite{Li}]\label{continuity}
Suppose that $(l^i,r^i)$ satisfy Assumption $\ref{asslr}$, $i=1,2$.
Given  $s^i\in C[0,T]$, let $(x^i,K^i)$ be the solution to the Skorokhod problem $\mathbb{SP}_{l^i}^{r^i}(s^i)$. Then, we have
\begin{equation}\label{diffk}
\sup_{t\in[0,T]}\left|K^1_t-K^2_t\right|
\leq \frac{C}{c}\sup_{t\in[0,T]}\left|s^1_t-s^2_t\right|+\frac{1}{c}(L_T\vee R_T),
\end{equation}
where
\begin{align*}
\bar{L}_T=\sup_{(t,x)\in[0,T]\times \mathbb{R}}\left|{l}^1(t,x)-{l}^2(t,x)\right|,\
\bar{R}_T=\sup_{(t,x)\in[0,T]\times \mathbb{R}}\left|{r}^1(t,x)-{r}^2(t,x)\right|.
\end{align*}
\end{proposition}

The following proposition provides the oscillation of $K$ in the closed  time interval $[\theta_1,\theta_2]\subset [0,T]$.
\begin{proposition}[\cite{Li}]\label{oscillation}
Given $s\in C[0,T]$, for any $t\geq 0 $, let $\phi_t,\psi_t$ be the unique solution to the following equations, respectively 
$$l(t,s_t+x)=0,\ r(t,s_t+x)=0.$$
 Then, the oscillation of $K$ can be dominated by the oscillation of $\phi, \psi$ on closed interval $[\theta_1,\theta_2]$, i.e.,
$$\sup_{t_1,t_2\in[\theta_1,\theta_2]}\left|K_{t_1}-K_{t_2}\right|\leq \sup_{s,u\in[\theta_1,\theta_2]}\left|\phi_{s}-\phi_{u}\right|+\sup_{s,u\in[\theta_1,\theta_2]}\left|\psi_{s}-\psi_{u}\right|. $$
\end{proposition}

\section{Main result}

In this paper, the objective is to investigate the well-posedness of the doubly mean reflected MFBSDE \eqref{BSDEDMR} with non-Lipschitz generator.  We need the following  assumptions for the terminal value $\xi$ and the generator $f$.  
\begin{assumption}\label{assf}
\begin{itemize}
\item[(i)] The terminal value $\xi\in\mathcal{L}^2(\mathcal{F}_T)$ satisfies $\mathbf{E}\left[L(T,\xi)\right]\leq 0\leq \mathbf{E}\left[R(T,\xi)\right]$.
\item[(ii)]  For any $t\in[0,T]$, $ y_1,y_2\in \mathbb{R}$, $\mu_1, \mu_2\in\mathcal{P}_{1}\left(\mathbb{R}\right)$,  $z_1,z_2\in \mathbb{R}^d$, $\nu_1,\nu_2\in\mathcal{P}_{1}\left(\mathbb{R}^d\right)$, there exists a constant $\lambda>0$ such that $\mathbf{E}\left[\int_{0}^{T}\left|f(t,0,\delta_0,0,\delta_0)\right|^2dt\right]<\infty$ and $\mathbf{P}$-$a.s.$, 
\begin{equation*}\begin{split}
 &\left|f(t,y_1,\mu_1,z_1,\nu_1)-f(t,y_2,\mu_2,z_2,\nu_2)\right|^2\\
 \leq& \rho\left(\left|y_1-y_2\right|^2+d_1^2\left(\mu_1,\mu_2\right)\right)+\lambda^2\left(\left|z_1-z_2\right|^2+d_1^2\left(\nu_1,\nu_2\right)\right),  
\end{split}  
\end{equation*}
where $\rho:[0,\infty) \rightarrow [0,\infty) $ is a continuous, non-decreasing and concave function satisfying $\rho(0)=0,\ \rho(r)>0$ for $r>0$, and $\int_{0+}\frac{dr}{\rho(r)}=\infty$.
\end{itemize}     
\end{assumption}
\begin{remark}\label{concave}
Without loss of generality, assume that $\rho(r)\leq \beta\left(1+\left|r\right|\right)$ for some constant $\beta>1$. Therefore, we have  $$\left|f(r,y,\mu,z,\nu)-f(r,0,\delta_{0},0,\delta_{0})\right|^2\leq \beta+\beta\left|y\right|^2+\beta d_1^2(\mu,\delta_0)+\lambda^2\left|z\right|^2+\lambda^2d_1^2\left(\nu,\delta_0\right).$$   
\end{remark}

The loss functions $L,R:\Omega\times [0,T]\times\mathbb{R}\rightarrow \mathbb{R}$ are measurable maps with respect to $\mathcal{F}_T\times \mathcal{B}([0,T])\times \mathcal{B}(\mathbb{R})$ satisfying the following conditions.

\begin{assumption}\label{assLR}
\begin{itemize}
\item[(i)] For any fixed $(\omega,x)\in \Omega\times\mathbb{R}$, $L(\omega, \cdot, x), R(\omega, \cdot, x)$ are continuous.
\item[(ii)]  
There exists a constant $M>0$ such that $$\mathbf{E}\left[\sup_{t\in[0,T]}\left|L(t,0)\right|+\sup_{t\in[0,T]}\left|R(t,0)\right|\right]\leq M.$$
\item[(iii)] For any fixed $(\omega,t)\in \Omega\times [0,T]$, $L(\omega,t,\cdot),R(\omega,t,\cdot)$ are strictly increasing and bi-Lipschitz continuous. That is, there  exist two constants $c,C$ satisfying $0<c<C$ such that for any $x,y\in \mathbb{R}$,
\begin{align*}
&c\left|x-y\right|\leq \left|L(\omega,t,x)-L(\omega,t,y)\right|\leq C\left|x-y\right|,\\
&c\left|x-y\right|\leq \left|R(\omega,t,x)-R(\omega,t,y)\right|\leq C\left|x-y\right|.
\end{align*}
\item[(iv)] $\inf_{\omega,t,x} \left(R(\omega,t,x)-L(\omega,t,x)\right)>0$.
\end{itemize}
\end{assumption}

We are now ready to state the main result of this paper.
\begin{theorem}\label{th-1}
 Let Assumptions \ref{assf} and \ref{assLR} hold. Then, the doubly mean reflected MFBSDE \eqref{BSDEDMR} admits a unique  solution $(Y,Z,K)\in \mathcal{S}^2\times\mathcal{H}^2\times BV[0,T]$ . 
\end{theorem}

 We will use the Picard iteration method to  prove Theorem \ref{th-1}. More precisely, we first set $Y^0=\left(\mathbf{E}_{t}\left[\xi\right]\right)_{0\leq t \leq T}$, and then define the sequence of the triple of  stochastic processes $\left\{\left(Y^m,Z^m,K^m\right)\right\}_{m=1}^{\infty}$ recursively  by solving the doubly mean reflected MFBSDEs:
 \begin{equation}\label{nonlinearyz-1}
\begin{cases}
Y^m_t=\xi+\int_t^T f^m(s,Z^m_s,\P_{Z^m_s})ds-\int_t^T Z^m_s dB_s+K^m_T-K^m_t, \\
\mathbf{E}[L(t,Y^m_t)]\leq 0\leq \mathbf{E}[R(t,Y^m_t)], \\
K^m_t=K^{R,m}_t-K^{L,m}_t,\ K^{R,m},K^{L,m}\in I[0,T], \\
\int_0^T \mathbf{E}[R(t,Y^m_t)]dK_t^{R,m}=\int_0^T \mathbf{E}[L(t,Y^m_t)]dK^{L,m}_t=0,
\end{cases}
\end{equation}
where for any $s\in[0,T]$, $z\in\mathbb{R}^d$ and $\nu\in\mathcal{P}_1(\mathbb{R}^d)$,
\begin{align}\label{fm}
    f^m(s,z,\nu)=f(s,Y^{m-1}_s,\P_{Y^{m-1}_s},z,\nu).
\end{align}
By Theorem 3.5 in \cite{LS}, for any $m=1,2,\cdots$, the doubly mean reflected MFBSDE \eqref{nonlinearyz-1} admits a  unique  solution $\left(Y^{m},Z^{m},K^{m}\right)\in \mathcal{S}^2\times\mathcal{H}^2\times BV[0,T]$. We will show that the sequence $\left\{\left(Y^m,Z^m,K^m\right)\right\}_{m=1}^{\infty}$ converges under the appropriate norm and the triple of limit processes is the solution to the doubly mean reflected MFBSDE \eqref{BSDEDMR}.

\subsection{A priori estimates}
This section is devoted to some uniform estimates for $\left\{\left(Y^m,Z^m,K^m\right)\right\}_{m=1}^{\infty}$ defined by \eqref{nonlinearyz-1}. 
\begin{lemma}\label{le-1}
Let Assumptions \ref{assf} and \ref{assLR} hold. Then for any $m\geq 1$,
\begin{equation*}
\mathbf{E}\left[\sup_{t\in[0,T]}\left|Y_{t}^{m}\right|^2+\int_{0}^{T}\left|Z_{t}^{m}\right|^2dt\right]\leq P(\lambda,\beta, C,c,M,T)\mathbf{E}\left[1+\left|\xi\right|^2+\int_{0}^{T}\left|f(s,0,\delta_0,0,\delta_0)\right|^2 ds\right]
\end{equation*}
\end{lemma}
\begin{proof}
 For any $t\in[0,T]$, by Theorem 3.1 in \cite{BLP}, the following $\text{MFBSDE}$
 \begin{equation*}\label{ym}
  y_t^m=\xi+\int_{t}^{T} f^m(r,z^m_r,\P_{z^m_r})dr- \int_{t}^{T}z_r^mdB_r 
 \end{equation*}
 admits a unique solution $(y^{m},z^{m})\in \mathcal{S}^2\times\mathcal{H}^2$. 
 Since $K^m$ is a deterministic function of bounded variation, $\left(Y^m-(K^m_{T}-K^m), Z^m\right)$ is again an $\mathcal{S}^2\times\mathcal{H}^2$-solution to MFBSDE \eqref{ym}, which implies that
 \begin{equation}\label{YmZm}
   Y_{t}^{m}=y_{t}^{m}+K_{T}^{m}-K_{t}^{m},\ Z_t^{m}=z_t^{m}.
 \end{equation}
 It follows that for any $t\in[0,T]$, 
 \begin{equation}\label{Ymym}
  \mathbf{E}[\left|Y_{t}^{m}\right|^2]\leq 2\mathbf{E}[\left|y_{t}^m\right|^2]+2\left|K_{T}^{m}-K_{t}^{m}\right|^2.
  \end{equation}
  According to Theorem 3.9 and Proposition 4.2 in \cite{L}, we know that the solution $K^m$ of \eqref{nonlinearyz-1} on the time interval $[0,T]$ 
has the following representation $$K^m_t=\bar{K}^m_{T}-\bar{K}^m_{T-t},$$ where $\bar{K}^m$ is the second component of the solution to the Skorokhod problem $\mathbb{SP}_{\bar{l}^m}^{\bar{r}^m}(\bar{s}^m)$. Here, for any $t\in[0,T]$ and any positive integer $m$, we have
\begin{equation}\label{smlmrm}\begin{split}
 &\bar{s}^m_{t}=\mathbf{E}[y^m_{T-t}],\\ 
 &\bar{l}^m(t,x)=\E\left[L\left(T-t,y_{T-t}^m-\E[y_{T-t}^m]+x\right)\right], \\ &\bar{r}^m(t,x)=\E\left[R\left(T-t,y_{T-t}^m-\E[y_{T-t}^m]+x\right)\right].   
\end{split}
\end{equation}

Let  $\phi^m$, $\psi^m$ be the unique solutions to the following equations, respectively, $$\bar{l}^m\left(t,\bar{s}_t^m+x\right)=0,\ \bar{r}^m\left(t,\bar{s}_t^m+x\right)=0.$$
Additionally, we introduce another new  Skorokhod problem $\mathbb{SP}_{\bar{l}^0}^{\bar{r}^0}(\bar{s}^{0})$, where for any $t\in[0,T]$
   $$
  \begin{aligned}
   \bar{s}_{t}^{0}=\mathbf{E}\left[\xi\right],\ \bar{l}^{0}(t,x)=\mathbf{E}\left[L(T-t,x)\right],\ \bar{r}^{0}(t,x)=\mathbf{E}\left[R(T-t,x)\right].
  \end{aligned}
  $$
 Let $\phi^{0}$, $\psi^{0}$ be the unique solutions to the following equations, respectively, $$\bar{l}^0\left(t,\bar{s}_t^{0}+x\right)=0,\ \bar{r}^0\left(t,\bar{s}_t^{0}+x\right)=0.$$ 
 Then applying Proposition \ref{oscillation}, we can get that for each $t\in[0,T]$,
  \begin{equation}\begin{split}\label{3-0}
   &\left|K_{T}^{m}-K_{t}^{m}\right|\leq\sup_{p_1,p_2\in[t,T]}\left|K_{p_1}^{m}-K_{p_2}^{m}\right|\\
   =&\sup_{p_1,p_2\in[t,T]}\left|\bar{K}_{T}^{m}-\bar{K}_{T-p_1}^{m}-\bar{K}_{T}^{m}+\bar{K}_{T-p_2}^{m}\right|\\ 
  \leq &\sup_{t_1,t_2\in[0,T-t]}\left|\bar{K}_{t_1}^{m}-\bar{K}_{t_2}^{m}\right|\leq \sup_{r,u\in[0,T-t]}\left|\phi_r^m-\phi_u^m\right|+\sup_{r,u\in[0,T-t]}\left|\psi_r^m-\psi_u^m\right|\\ 
  \leq&2\left(\sup_{r\in[0,T-t]}\left|\phi_r^m-\phi_r^{0}\right|+\sup_{r\in[0,T-t]}\left|\psi_r^m-\psi_r^{0}\right|+\sup_{r\in[0,T-t]}\left|\phi_r^{0}\right|+\sup_{r\in[0,T-t]}\left|\psi_r^{0}\right|\right).
  \end{split}   
  \end{equation}
Recalling the definition of $\bar{l}^m$ and $\bar{l}^0$, it is easy to check that 
   $$
  \begin{aligned}
  c\left|\bar{s}_{r}^m+\phi_r^m-(\bar{s}_{r}^{0}+\phi_r^0)\right|
  \leq&\left|\bar{l}^{m}(r,\bar{s}_{r}^m+\phi_r^m)-\bar{l}^{m}(r,\bar{s}_{r}^{0}+\phi_r^0)\right|\\
  = &\left|\bar{l}^{m}(r,\bar{s}_{r}^{0}+\phi_r^0)-\bar{l}^{0}(r,\bar{s}_{r}^{0}+\phi_r^0)\right| \\
  \leq& 2C\mathbf{E}\left[\left|y^{m}_{T-r}\right|\right],\\
  \end{aligned} 
   $$
It follows that
  \begin{equation}\begin{split}\label{3-1}
  \sup_{r\in[0,T-t]}\left|\phi_r^m-\phi_r^{0}\right|&\leq \frac{2C}{c}\sup_{r\in[0,T-t]}\mathbf{E}\left[\left|y^{m}_{T-r}\right|\right]+\sup_{r\in[0,T-t]}\left|\bar{s}_r^m-\bar{s}_r^{0}\right|\\
  &\leq \left(1+\frac{2C}{c}\right)\sup_{r\in[0,T-t]}\mathbf{E}\left[\left|y^{m}_{T-r}\right|\right]+\mathbf{E}\left[\left|\xi\right|\right].\\  \end{split}  
  \end{equation}
Similarly to \eqref{3-1}, we have 
\begin{equation}\label{3-2}
  \sup_{r\in[0,T-t]}\left|\psi_r^m-\psi_r^{0}\right| \leq \left(1+\frac{2C}{c}\right)\sup_{r\in[0,T-t]}\mathbf{E}\left[\left|y^{m}_{T-r}\right|\right]+\mathbf{E}\left[\left|\xi\right|\right].
  \end{equation}
 Thus, applying Assumption \ref{assLR} and combining \eqref{3-0}-\eqref{3-2}, we obtain that 
  \begin{equation}\label{3-5}
  \left|K_{T}^{m}-K_{t}^{m}\right|\leq 4\left(1+\frac{2C}{c}\right)\mathbf{E}\left[\sup_{r\in[t,T]}\left|y^{m}_{r}\right|\right]+4\left(1+\frac{2C}{c}\right)\mathbf{E}\left[\left|\xi\right|\right]+P(c,C,M),
  \end{equation} 
  where $P(c,C,M)=\frac{2C}{c}\E[\left|\xi\right|]+\frac{M}{c}$.
  
 By Proposition 2.2 in \cite{LLZ} and  Remark \ref{concave}, we obtain that for any $t\in[0,T]$,
 \begin{equation}
 \begin{split}\label{ym-1}
&\E\left[\sup_{r\in[t,T]}\left|y_r^m\right|^2\right]\\
\leq& P(\lambda,T)\E\left[\left|\xi\right|^2+\int_{t}^{T}\left|f(s,Y_s^{m-1},\P_{Y_s^{m-1}},0,\delta_0)\right|^2ds\right]\\
\leq& P(\lambda,T)\E\left[\left|\xi\right|^2+2\int_{t}^{T}\left|f(s,Y_s^{m-1},\P_{Y_s^{m-1}},0,\delta_0)-f(s,0,\delta_0,0,\delta_0)\right|^2+\left|f(s,0,\delta_0,0,\delta_0)\right|^2ds\right]\\
\leq& P(\lambda,T)\E\left[2\beta(T-t)+\left|\xi\right|^2+2\int_{t}^{T}\left|f(s,0,\delta_0,0,\delta_0)\right|^2ds+4\beta\int_{t}^{T}\left|Y_s^{m-1}\right|^2ds\right]
 \end{split}  
 \end{equation}
 Combining \eqref{3-5} and \eqref{ym-1}, we have
  \begin{equation}\label{3-6}
  \left|K_{T}^{m}-K_{t}^{m}\right|^2
  \leq P(\lambda,\beta,C,c,T,M)\E\left[1+\left|\xi\right|^2+\int_{t}^{T}\left|f(s,0,\delta_0,0,\delta_0)\right|^2ds+\int_{t}^{T}\left|Y_s^{m-1}\right|^2ds\right].   
  \end{equation} 
Plugging Equations \eqref{ym-1} and \eqref{3-6} into \eqref{Ymym} yields that 
\begin{equation*}
\label{3-7}
\mathbf{E} [\left|Y_t^m\right|^2]
\leq P(\lambda,\beta,C,c,T,M)\E\left[1+\left|\xi\right|^2+\int_{t}^{T}\left|f(s,0,\delta_0,0,\delta_0)\right|^2ds+\int_{t}^{T}\left|Y_s^{m-1}\right|^2ds\right].     
\end{equation*}

Set $$p(t)=P(\lambda,\beta,C,c,T,M)\E\left[1+\left|\xi\right|^2+\int_{t}^{T}\left|f(s,0,\delta_0,0,\delta_0)\right|^2ds\right]e^{P(\lambda,\beta,C,c,T,M)(T-t)}.$$
It is easy to check that $p(\cdot)$ is the solution to the following ODE:
$$p(t)=P(\lambda,\beta,C,c,T,M)\E\left[1+\left|\xi\right|^2+\int_{t}^{T}\left|f(s,0,\delta_0,0,\delta_0)\right|^2ds+\int_{t}^{T}p(s)ds\right].$$
Hence, by recurrence we get that for any $m\geq1$ and for each $t\in[0,T]$,
$$\begin{aligned}
 \E[\left|Y_t^m\right|^2]\leq p(t) \leq \max_{t\in[0,T]}p(t)=p(0)=P(\lambda,\beta,C,c,T,M)\E\left[1+\left|\xi\right|^2+\int_{t}^{T}\left|f(s,0,\delta_0,0,\delta_0)\right|^2ds\right].
\end{aligned}$$
Recalling \eqref{YmZm},\eqref{ym-1} and applying Proposition 2.2 in \cite{LLZ}, we obtain that
\begin{equation*}
\label{3-9}
\E\left[\int_{0}^{T}\left|Z_s^m\right|^2ds\right]\leq P(\lambda,\beta,C,c,T,M)\E\left[1+\left|\xi\right|^2+\int_{0}^{T}\left|f(s,0,\delta_0,0,\delta_0)\right|^2ds\right]   
\end{equation*}
Applying the Doob inequality, we finally derive that
\begin{equation*}
\begin{split}
 &\mathbf{E} \left[\sup_{t\in[0,T]}\left|Y_t^m\right|^2\right]\leq P(\lambda,\beta, C,c,T,M )\mathbf{E}\left[1+\left|\xi\right|^2+\int_{0}^{T}\left|f(s,0,\delta_0,0,\delta_0)\right|^2ds\right].
\end{split}    
\end{equation*}
The proof is complete.
\end{proof}

The following lemma indicates that $\{Y^m\}_{m=1}^\infty$ and $\{Z^m\}_{m=1}^\infty$ are Cauchy sequneces in $\mathcal{S}^2$ and $\mathcal{H}^2$, respectively.
\begin{lemma}\label{le-2}
 Let Assumptions \ref{assf} and \ref{assLR} hold. Then, we have
 $$\mathbf{E}\left[\sup_{t\in[0,T]}\left|Y_t^m-Y_t^n\right|^2\right]\rightarrow 0,\ \mathbf{E}\left[\int_{0}^{T}\left|Z_{t}^m-Z_{t}^{n}\right|^2dt\right]\rightarrow 0,\ \text{as}\ m,n\rightarrow \infty.$$
\end{lemma}
\begin{proof}
The proof is divided into two steps.\\
(i) $\textbf{Convergence of Y}$. 
Similar as the proof of Lemma \ref{le-1}, $(Y^i,Z^i)$ has the following representation, i.e., for $i=m,m+p$ with   $m,p\in\mathbb{N}$,
\begin{equation}\begin{split}\label{Yi}
&Y_s^i=y_s^i+K^i_{T}-K^i_s=y_s^i+\bar{K}^{i}_{T-s},\ Z^i_s=z^i_s,\ s\in[0,T],\\
\end{split}   
\end{equation}
where $y_{\cdot}^i$ is the first component of the solution to the MFBSDE with terminal value $\xi$ and generator $f^i$ (see \eqref{fm}). Similar to Lemma $\ref{le-1}$, we know that $\bar{K}^i$ is the solution of Skorokhod problem $\mathbb{SP}_{\bar{l}^i}^{\bar{r}^i}(\bar{s}^i)$, where $\bar{s}^i$, $\bar{l}^i$ and $\bar{r}^i$ are given in \eqref{smlmrm}. 
 Applying Theorem \ref{4-2} and Jensen's inequality, we have 
\begin{equation}\label{ymymp}
     \begin{split}
      &\E\left[\sup_{s\in[t,T]}\left|y_s^m-y_s^{m+p}\right|^2\right]\\
      \leq &P(\lambda,T)\mathbf{E}\left[\int_t^T \left|f^m(s,z^m_s,\mathbf{P}_{z^m_s})-f^{m+p}(s,z^m_s,\mathbf{P}_{z^m_s})\right|^2ds\right]\\
      \leq& P(\lambda,T)\mathbf{E}\left[\int_{t}^{T}\rho\left(\left|Y_s^{m-1}-Y_{s}^{m+p-1}\right|^2+\mathbf{E}\left[\left|Y_s^{m-1}-Y_{s}^{m+p-1}\right|^2\right]\right)ds\right]\\
      \leq& P(\lambda,T)\int_{t}^{T}\rho\left(2\mathbf{E}\left[\left|Y_s^{m-1}-Y_{s}^{m+p-1}\right|^2\right]\right)ds\\
      \leq& P\left(\lambda, T\right)\int_{t}^{T}\rho\left(2\mathbf{E}\left[\sup_{v\in[s,T]}\left|Y_v^{m-1}-Y_{v}^{m+p-1}\right|^2\right]\right)ds.
\end{split}
 \end{equation}
 In view of Proposition \ref{continuity} , we know
\begin{equation*}\begin{split}
 \sup_{r \in[t, T]}\left|\bar{K}^m_{T-r}-\bar{K}^{m+p}_{T-r}\right|
 \leq \frac{C}{c}\sup_{r \in[0, T-t]}\left|\bar{s}^m_r-\bar{s}^{m+p}_r\right|+\frac{1}{c}(\bar{L}_{T-t} \vee \bar{R}_{T-t}),
 \end{split}    
\end{equation*}
where $$\bar{L}_{T-t}=\sup_{(v,x)\in[0, T-t]\times \mathbb{R}}|\bar{l}^m(v,x)-\bar{l}^{m+p}(v,x)|,
\bar{R}_{T-t}=\sup_{(v,x)\in[0, T-t]\times \mathbb{R}}|\bar{r}^m(v,x)-\bar{r}^{m+p}(v,x)|.$$
Recalling that  $\bar{s}^i_r=\mathbf{E}[y^i_{T-r}]$, we have 
\begin{align*}
	\sup_{r\in[0, T-t]}|\bar{s}^m_r-\bar{s}^{m+p}_r|
    \leq \sup_{r \in[t, T]}\mathbf{E}\left[\left|y^m_{r}-y^{m+p}_{r}\right|\right].
	\end{align*}
Applying Assumption \ref{assLR}, we obtain that 
\begin{equation*}
	\bar{L}_{T-t}\leq 2C\sup_{r \in[t, T]}\mathbf{E}\left[\left|y^m_{r}-y^{m+p}_{r}\right|\right],\ \bar{R}_{T-t}\leq 2C\sup_{r \in[t, T]}\mathbf{E}\left[\left|y^m_{r}-y^{m+p}_{r}\right|\right].
\end{equation*}
 The above analysis implies that 
\begin{equation}\label{kmkmp}
	\sup_{r \in[t, T]}\left|\bar{K}^m_{T-r}-\bar{K}^{m+p}_{T-r}\right|
	\leq \frac{3C}{c}\sup_{r \in[t, T]}\mathbf{E}\left[\left|y^m_{r}-y^{m+p}_{r}\right|\right].
\end{equation}
 According to Equations \eqref{Yi}, \eqref{ymymp} and \eqref{kmkmp}, we have
 \begin{equation}\begin{split}\label{Ys}
   \mathbf{E}\left[\sup_{s\in[t,T]}\left|Y^m_s-Y^{m+p}_s\right|^2\right]
   &\leq P\left(\lambda, C,c,T\right)\int_{t}^{T}\rho\left(2\mathbf{E}\left[\sup_{v\in[s,T]}\left|Y_v^{m-1}-Y_{v}^{m+p-1}\right|^2\right]\right)ds.
 \end{split}
 \end{equation}
 
Now, we define $$u_{m,p}(t)=2\mathbf{E}\left[\sup_{s\in[t,T]}\left|Y^m_s-Y^{m+p}_s\right|^2\right],$$ which is uniformly bounded by Lemma $\ref{le-1}$. It follows from \eqref{Ys} that 
$$u_{m,p}(t)
\leq P\left(\lambda, C,c,T\right)\int_{t}^{T}\rho\left(u_{m-1,p}(r)\right)dr.$$Set $u_{m}(t)=\sup_{p\geq1}u_{m,p}(t)$ and $ \alpha(t)=\lim\sup_{m\rightarrow\infty}u_{m}(t)$. It follows that 
$$\alpha(t)\leq P\left(\lambda, C,c,T\right)\int_{t}^{T}\rho\left(\alpha(r)\right)dr,$$
which together with backward Bihari's inequality yields that $\alpha(t)=0,\ t\in[0,T]$. Consequently, we get that 
$$\lim_{m,n\rightarrow \infty}\mathbf{E}\left[\sup_{t\in[0,T]}\left|Y_{t}^m-Y_{t}^{n}\right|^2\right]=0.$$
(ii) $\textbf{Convergence of Z}$. Applying Theorem $\ref{4-2}$ and \eqref{ymymp}, we have  
\begin{equation}
    \begin{split}
    \E\left[\int_{0}^{T}\left|Z_s^m-Z_s^{m+p}\right|^2 ds\right]&\leq P(\lambda,T)\E\left[\int_{0}^{T}\left|f^m(s,Z^m_s,\mathbf{P}_{Z^m_s})-f^{m+p}(s,Z^m_s,\mathbf{P}_{Z^m_s})\right|^2ds\right]\\
    &\leq P\left(\lambda, T\right)\int_{0}^{T}\rho\left(2\mathbf{E}\left[\sup_{v\in[s,T]}\left|Y_v^{m-1}-Y_{v}^{m+p-1}\right|^2\right]\right)ds.
    \end{split}
\end{equation}
Therefore,  we can conclude that 
$$\lim_{m,n\rightarrow \infty}\mathbf{E}\left[\int_{0}^{T}\left|Z_{t}^m-Z_{t}^{n}\right|^2dt\right]=0.$$
The proof is complete.
\end{proof}
\subsection{Proof of Theorem \ref{th-1}}
Now we are in a position to give the proof of our main result. 

\begin{proof}
(i) $\textbf{Existence.}$ In light of Lemma $\ref{le-2}$, there is a pair of progressively measurable processes $(Y,Z)$ such that
\begin{equation*}\label{YZ}
 \lim_{m\rightarrow \infty} \left(\left \|Y^m-Y  \right \|_{\mathcal{S}^2}+\left \|Z^m-Z  \right \|_{\mathcal{H}^2}\right)=0.   
\end{equation*}

By a standard argument, we have
\begin{equation*}\label{fmf}
\lim_{m\rightarrow \infty}\mathbf{E}\left[\int_{0}^{T}\left|f(r,Y^{m-1}_r,\P_{Y^{m-1}_r}, Z_r^m,\P_{Z_r^m})-f(r,Y_r,\P_{Y_r}, Z_r,\P_{Z_r})\right|^2dr\right]=0   
\end{equation*}
Set $K_t:=Y_t-Y_0+\int_{0}^{T}f\left(s, Y_s,\P_{Y_s}, Z_s,\P_{Z_s}\right)ds-\int_{0}^{T}Z_sdB_s$. Consequently, 
\begin{equation*}\label{Km}
\lim_{m \rightarrow \infty}\mathbf{E}\left[\sup _{t \in[0,T]}\left|K_t-K_t^{m}\right|\right]=0.
\end{equation*}
In particular, we have $K_t:=\lim _{m \rightarrow \infty} K_t^{m}=\lim _{m \rightarrow \infty} \mathbf{E}\left[K_t^{m}\right]=\mathbf{E}\left[K_t\right]$ and then $K$ is a continuous function with bounded function. It remains to show the Skorokhod condition holds. The proof  is similar to the one for Theorem 4.4 in \cite{LS}, so we omit it here.

(ii) $\textbf{Uniqueness.}$ The proof is similar with the one for Theorem 2.3 in \cite{CZ}. For readers' convenience, we give a short proof here.  Assume that $(Y',Z',K')$ is another solution to the doubly mean reflected $\text{MFBSDE}$. Then, similar to proof for Lemma $\ref{le-2}$, we could deduce for any $t\in[0,T]$,
\begin{equation*}
    \mathbf{E}\left[\sup_{r\in[t,T]}\left|Y_r-Y'_r\right|^2\right]\leq P(C,c,\lambda,T)\int_{t}^{T}\rho\left(2\mathbf{E}\left[\sup_{v\in[r,T]}\left|Y_v-Y'_v\right|^2\right]\right)dr,
\end{equation*}
which yields that $Y=Y'$. Then both $(Y,Z,K)$ and $(Y',Z',K')$ can be viewed as the solution to the following doubly mean reflected MFBSDE with Lipschitz generator 
\begin{equation*}
\begin{cases}
\bar{Y}_t=\xi+\int_t^T f(s,Y_s,\P_{Y_s},\bar{Z}_s,\P_{\bar{Z}_s})ds-\int_t^T \bar{Z}_s dB_s+\bar{K}_T-\bar{K}_t, \\
\mathbf{E}[L(t,\bar{Y}_t)]\leq 0\leq \mathbf{E}[R(t,\bar{Y}_t)], \\
\bar{K}_t=\bar{K}^R_t-\bar{K}^L_t,\ \bar{K}^R,\bar{K}^L\in I[0,T], \\
\int_0^T \mathbf{E}[R(t,\bar{Y}_t)]d\bar{K}_t^R=\int_0^T \mathbf{E}[L(t,\bar{Y}_t)]d\bar{K}^L_t=0.
\end{cases}
\end{equation*}
According to Theorem 3.5 in \cite{LS}, we obtain that $(Z,K)=(Z',K')$. The proof is complete.
\end{proof}

\section*{Acknowledgement}
    This work was supported  by the National Natural Science Foundation of China (No. 12301178), the Natural Science Foundation of Shandong Province for Excellent Young Scientists Fund Program (Overseas) (No. 2023HWYQ-049), the Fundamental Research Funds for the Central Universities and  the Qilu Young Scholars Program of Shandong University. 
\renewcommand\thesection{Appendix A}
\section{\texorpdfstring{$\text{\small A priori estimates for Mean-field BSDEs}$}{A priori estimates for Mean-field BSDEs}}\label{append:A}
\renewcommand\thesection{A}

In this section, we consider the following MFBSDEs:
\begin{equation}\label{MFBSDE}
Y_t=\xi+\int_t^T f(s,Y_s,\P_{Y_s},Z_s,\P_{Z_s})ds-\int_t^T Z_s dB_s,  
\end{equation}
where the generator $f$ and the terminal value $\xi$ satisfy the following conditions.

\begin{assumption}\label{assfL}
\begin{itemize}
\item[(i)] The terminal value $\xi\in\mathcal{L}^2(\mathcal{F}_T)$.
\item[(ii)]  For any $t\in[0,T]$, $ y_1,y_2\in \mathbb{R}$, $\mu_1, \mu_2\in\mathcal{P}_{1}\left(\mathbb{R}\right)$,  $z_1,z_2\in \mathbb{R}^d$, $\nu_1,\nu_2\in\mathcal{P}_{1}\left(\mathbb{R}^d\right)$, there exists a constant $\lambda>0$ such that $\mathbf{E}\left[\int_{0}^{T}\left|f(t,0,\delta_0,0,\delta_0)\right|^2dt\right]<\infty$ and $\mathbf{P}$-$a.s.$, 
\begin{equation*}
 \left|f(t,y_1,\mu_1,z_1,\nu_1)-f(t,y_2,\mu_2,z_2,\nu_2)\right|
 \leq \lambda\left(\left|y_1-y_2\right|+d_1\left(\mu_1,\mu_2\right)+\left|z_1-z_2\right|+d_1\left(\nu_1,\nu_2\right)\right).
 \end{equation*}
\end{itemize}
\end{assumption}


Motivated by the proof of Proposition 2.2 in \cite{LLZ}, we have the following a priori estimates for the solutions to MFBSDE \eqref{MFBSDE}.
\begin{theorem}\label{4-2}
 For $i=1,2$, assume that $(\xi^i,f^i)$ satisfy Assumption \ref{assfL} and $(Y^i,Z^i)\in\mathcal{S}^2\times\mathcal{H}^2$ are the solution to MFBSDE \eqref{MFBSDE} with parameters $(\xi^i,f^i)$. Then, for any $ t\in[0,T]$, we have
 \begin{equation*}
  \E\left[\sup_{s\in[t,T]}\left|\Delta Y_s\right|^2+\int_{t}^{T}\left|\Delta Z_s\right|^2 ds\right]\leq P(\lambda,T)\E\left[\left|\Delta\xi\right|^2+\int_{t}^{T}\left|\Delta f(s,Y_s^1,\P_{Y_s^1},Z_s^1,\P_{Z_s^1})\right|^2ds\right],       
 \end{equation*}
 where $$\Delta Y:=Y^1-Y^2,\ \Delta Z:=Z^1-Z^2,\ \Delta \xi:=\xi^1-\xi^2,\ \Delta f:=f^1-f^2$$
\end{theorem}
\begin{proof}
 According to \eqref{MFBSDE} and the Ito isometry, for any $t\in[0,T]$, we have,
 \begin{equation}
     \begin{split}\label{4-3}
        &\E_t\left[\left|\Delta Y_t\right|^2+\int_{t}^{T}\left|\Delta Z_s\right|^2 ds\right]\\
        \leq& 2\E_t\left[\left|\Delta\xi\right|^2+(T-t)\int_{t}^{T}\left| f^1(s,Y_s^1,\P_{Y_s^1},Z_s^1,\P_{Z_s^1})-f^2(s,Y_s^2,\P_{Y_s^2},Z_s^2,\P_{Z_s^2})\right|^2ds\right]\\
        \leq& 2\E_t\left[\left|\Delta\xi\right|^2+2(T-t)\int_{t}^{T}\left|\Delta f(s,Y_s^1,\P_{Y_s^1},Z_s^1,\P_{Z_s^1})\right|^2ds\right]\\
        &+4(T-t)\E_t\left[\int_{t}^{T}\left| f^2(s,Y_s^1,\P_{Y_s^1},Z_s^1,\P_{Z_s^1})-f^2(s,Y_s^2,\P_{Y_s^2},Z_s^2,\P_{Z_s^2})\right|^2ds\right]\\
         \leq& 2\E_t\left[\left|\Delta\xi\right|^2+2(T-t)\int_{t}^{T}\left|\Delta f(s,Y_s^1,\P_{Y_s^1},Z_s^1,\P_{Z_s^1})\right|^2ds\right]\\
        &+16\lambda^2(T-t)\E_t\left[\int_{t}^{T}\left|\Delta Y_s\right|^2+\E\left[\left|\Delta Y_s\right|^2\right]+ \left|\Delta Z_s\right|^2+ \E\left[\left|\Delta Z_s\right|^2\right]ds\right].
     \end{split}
 \end{equation}
 Taking expectations on both sides of \eqref{4-3}, we have 
 \begin{equation*}
     \begin{split}\label{4-4}
        \E\left[\left|\Delta Y_t\right|^2+\int_{t}^{T}\left|\Delta Z_s\right|^2 ds\right]
         \leq& 2\E\left[\left|\Delta\xi\right|^2+2(T-t)\int_{t}^{T}\left|\Delta f(s,Y_s^1,\P_{Y_s^1},Z_s^1,\P_{Z_s^1})\right|^2ds\right]\\
        &+32\lambda^2(T-t)\E\left[\int_{t}^{T}\left|\Delta Y_s\right|^2+\left|\Delta Z_s\right|^2ds\right].
     \end{split}
 \end{equation*}
 Then, for $\delta>0$ small enough such that $32\delta \lambda^2<1$, it follows from Gronwall's inequality that, for any $t\in[T-\delta,T]$,
 \begin{equation}
     \begin{split}\label{4-5}
        &\E\left[\left|\Delta Y_t\right|^2+\int_{t}^{T}\left|\Delta Z_s\right|^2 ds\right]\leq P(\lambda,T)\E\left[\left|\Delta\xi\right|^2+\int_{t}^{T}\left|\Delta f(s,Y_s^1,\P_{Y_s^1},Z_s^1,\P_{Z_s^1})\right|^2ds\right].
     \end{split}
 \end{equation}
 Therefore, for $t\in [T-\delta,T]$, applying Gronwall's inequality and combining \eqref{4-3} and \eqref{4-5}, we get 
  \begin{equation}
     \begin{split}\label{4-6}
        \left|\Delta Y_t\right|^2+\E_t\left[\int_{t}^{T}\left|\Delta Z_s\right|^2 ds\right]
         &\leq P(\lambda,T)\E\left[\left|\Delta\xi\right|^2+\int_{t}^{T}\left|\Delta f(s,Y_s^1,\P_{Y_s^1},Z_s^1,\P_{Z_s^1})\right|^2ds\right]\\
         &+ P(\lambda,T)\E_t\left[\left|\Delta\xi\right|^2+\int_{t}^{T}\left|\Delta f(S,Y_s^1,\P_{Y_s^1},Z_s^1,\P_{Z_s^1})\right|^2ds\right].
     \end{split}
 \end{equation}
 Consequently, using Doob's inequality, we obtain that 
 \begin{equation*}
     \begin{split}\label{4-7}
        \E_t\left[\sup_{s\in[t,T]}\left|\Delta Y_s\right|^2\right]
         \leq& P(\lambda,T)\E\left[\left|\Delta\xi\right|^2+\int_{t}^{T}\left|\Delta f(s,Y_s^1,\P_{Y_s^1},Z_s^1,\P_{Z_s^1})\right|^2ds\right]\\
         &+ P(\lambda,T)\E_t\left[\sup_{s\in[t,T]}\E_s\left[\left|\Delta\xi\right|^2+\int_{t}^{T}\left|\Delta f(s,Y_s^1,\P_{Y_s^1},Z_s^1,\P_{Z_s^1})\right|^2ds\right]\right]\\
         \leq& P(\lambda,T)\E\left[\left|\Delta\xi\right|^2+\int_{t}^{T}\left|\Delta f(s,Y_s^1,\P_{Y_s^1},Z_s^1,\P_{Z_s^1})\right|^2ds\right]\\
         &+ P(\lambda,T)\E_t\left[ \left|\Delta\xi\right|^2+\int_{t}^{T}\left|\Delta f(s,Y_s^1,\P_{Y_s^1},Z_s^1,\P_{Z_s^1})\right|^2ds\right],
     \end{split}
 \end{equation*}
 which together with \eqref{4-6} yields that 
 \begin{equation}
     \begin{split}\label{4-8}
        \E_t\left[\sup_{s\in[t,T]}\left|\Delta Y_s\right|^2+\int_{t}^{T}\left|\Delta Z_s\right|^2 ds\right]
         \leq& P(\lambda,T)\E\left[\left|\Delta\xi\right|^2+\int_{t}^{T}\left|\Delta f_s(Y_s^1,\P_{Y_s^1},Z_s^1,\P_{Z_s^1})\right|^2ds\right]\\
         +& P(\lambda,T)\E_t\left[\left|\Delta\xi\right|^2+\int_{t}^{T}\left|\Delta f(s,Y_s^1,\P_{Y_s^1},Z_s^1,\P_{Z_s^1})\right|^2ds\right].
     \end{split}
 \end{equation}
 Next, for $t\in[T-2\delta,T-\delta]$, by a similar analysis as \eqref{4-8}, we obtain that 
\begin{equation*}
     \begin{split}\label{4-9}
        &\E_t\left[\sup_{s\in[t,T-\delta]}\left|\Delta Y_s\right|^2+\int_{t}^{T-\delta}\left|\Delta Z_s\right|^2 ds\right]\\
         \leq& P(\lambda,T)\E\left[\left|\Delta Y_{T-\delta}\right|^2+\int_{t}^{T-\delta}\left|\Delta f(s,Y_s^1,\P_{Y_s^1},Z_s^1,\P_{Z_s^1})\right|^2ds\right]\\
         &+ P(\lambda,T)\E_t\left[\left|\Delta Y_{T-\delta}\right|^2+\int_{t}^{T-\delta}\left|\Delta f(s,Y_s^1,\P_{Y_s^1},Z_s^1,\P_{Z_s^1})\right|^2ds\right],
     \end{split}
 \end{equation*}
 which, together with \eqref{4-8} yields that, for any $t\in[T-2\delta, T]$, we have
 \begin{equation*}
     \begin{split}\label{4-10}
        \E_t\left[\sup_{s\in[t,T]}\left|\Delta Y_s\right|^2+\int_{t}^{T}\left|\Delta Z_s\right|^2 ds\right]
         \leq& P(\lambda,T)\E\left[\left|\Delta\xi\right|^2+\int_{t}^{T}\left|\Delta f(s,Y_s^1,\P_{Y_s^1},Z_s^1,\P_{Z_s^1})\right|^2ds\right]\\
         +& P(\lambda,T)\E_t\left[\left|\Delta\xi\right|^2+\int_{t}^{T}\left|\Delta f(s,Y_s^1,\P_{Y_s^1},Z_s^1,\P_{Z_s^1})\right|^2ds\right].\\
     \end{split}
 \end{equation*}
 Iterating the above argument $N$ times, with $N\delta\geq T$, we have 
 \begin{equation*}
     \begin{split}\label{4-11}
        \E_t\left[\sup_{s\in[t,T]}\left|\Delta Y_s\right|^2+\int_{t}^{T}\left|\Delta Z_s\right|^2 ds\right]
         \leq& P(\lambda,T)\E\left[\left|\Delta\xi\right|^2+\int_{t}^{T}\left|\Delta f(s,Y_s^1,\P_{Y_s^1},Z_s^1,\P_{Z_s^1})\right|^2ds\right]\\
         +& P(\lambda,T)\E_t\left[\left|\Delta\xi\right|^2+\int_{t}^{T}\left|\Delta f(s,Y_s^1,\P_{Y_s^1},Z_s^1,\P_{Z_s^1})\right|^2ds\right].
     \end{split}
 \end{equation*}
 The proof is complete.\end{proof}

\end{document}